\documentclass[twoside]{article}
\usepackage{amsmath, amsthm, amssymb, bm, graphicx, mathrsfs}
\usepackage{graphicx}
\usepackage{float}
\usepackage{xcolor}
\usepackage{subcaption}
\usepackage{xspace}
\usepackage{algorithm}
\usepackage{algorithmicx}
\usepackage{algpseudocode}
\usepackage{booktabs}
\usepackage{multirow}
\usepackage{fancyhdr}
\usepackage{comment}
\usepackage[utf8]{inputenc}
\usepackage[T1]{fontenc}
\usepackage{lmodern}
\usepackage[margin=1in]{geometry}
\usepackage{booktabs}
\usepackage{siunitx}
\usepackage{hyperref}
\hypersetup{colorlinks,allcolors=blue}

\newtheorem{theorem}{Theorem}

\newtheorem{lemma} [theorem]{Lemma}
\newtheorem{proposition} [theorem]{Proposition}
\newtheorem{appendix*}{Appendix}
\newtheorem{remark} [theorem]{Remark}
\newtheorem{example}[theorem]{Example}

\newtheorem{corollary}[theorem]{Corollary}

\makeatletter
\let\c@algorithm\c@theorem
\makeatother

\def\e{\epsilon}

\def\id{\text{id}}
\def\ind{\text{ind}}
\def\adj{\text{adj}}

\def\>{\right]}
\def\<{ \left[}

\def\y{\mathbf{y}}

\def\cone{\mathscr{C}}

\def\Zy{\mathcal{Z}_{\mathbf{y}}}

\begin{document}
		\title{Closed-Form Decomposition for Simplicial Cones and PDBarv Algorithm for Lattice Point Counting}
	\author{SIHAO TAO, GUOCE XIN* AND ZIHAO ZHANG}
        \date{}
	\maketitle

	\begin{abstract}
Counting lattice points within a rational polytope is a foundational problem with applications across mathematics and computer science. A key approach is Barvinok's algorithm, which decomposes the lattice point generating function of cones to that of unimodular cones. However, standard implementations face difficulties: the original primal method struggles with points on cone boundaries, while the alternative dual method can be slow for certain cone types.

This paper introduces two main contributions. First, We derive a closed-form expression for these generating functions using arbitrary lattice point decompositions, enabling more effective primal space decomposition. Second,  by decomposing both the cone and its dual cone starting from the side with a smaller index, we develop a novel algorithm called \textup{PDBarv}. This hybrid approach integrates the primal and dual Barvinok algorithms with a novel acceleration strategy, achieving an average computational performance improvement of over 20\% in dimension 5 and even better in higher dimensions.
\vspace{1em}\\
\noindent
    \begin{small}
    \end{small}


	\end{abstract}

\section{Introduction}
The problem of counting lattice points in rational polyhedra is a core issue in algebraic combinatorics and computational geometry, with applications in optimization, number theory, and algebraic geometry. This task is particularly difficult due to the complexity of high-dimensional polyhedra and the need for efficient algorithms to handle large-scale computations. Accurate lattice point counting is essential for solving various mathematical and computational problems, making it an area of active research.

In Euclidean space \(\mathbb{R}^d\), a polyhedron is defined as the intersection of a finite number of half-spaces, and when bounded, it is called a polytope. A shifted cone, which has the form \(\{v + k_1 w_1 + \dots + k_n w_n : k_1, \dots, k_n \geq 0\}\), is known as a simplicial cone if the vectors \(w_1, \dots, w_n\) are linearly independent. Our research focuses on counting lattice points in rational polyhedra, where all defining half-spaces are rational.

The problem of counting lattice points in rational polyhedra has a well-known solution in two dimensions using Pick's theorem \cite{Pick_thm}. This theorem states that for simple polygons with integer-coordinate vertices, the area \( S \) can be expressed as \( S = n + \frac{s}{2} - 1 \), where \( n \) is the number of interior lattice points and \( s \) is the number of boundary lattice points. However, such properties do not generally extend to higher-dimensional polyhedra.

Several efficient algorithms exist for computing lattice-point generating functions. A notable one is Barvinok's algorithm \cite{Primal-barvinok}, which runs in theoretical polynomial time when the dimension is fixed. Barvinok's algorithm reduces the problem of counting lattice points in a rational polyhedron to computing the generating functions of simplicial cones, leveraging Brion's theorem~\cite{Brion-Thm-and-dual-trick} and cone triangulations \cite{Triangulations-of-cone}. Barvinok demonstrated that generating functions of simplicial cones can be computed efficiently in polynomial time.

However, the original Barvinok algorithm, which operates in the primal space, faces challenges in handling lattice points on the faces of cones. To address these limitations, Barvinok et al. introduced the dual Barvinok algorithm \cite{Dual-barvinok}. This dual approach exploits cone duality and the Dual trick to eliminate lower-dimensional cones in the dual space after decomposition.
 While this method offers advantages, it can become significantly slower when the index of the primal cone is much smaller than that of the dual cone. While new methods have emerged, they bring with them new challenges, particularly in the realm of high-dimensional problems, which are inherently complex. See Section 2.

In this paper, we establish a closed-form formula for decomposing a simplicial cone in terms of generating functions rather than indicator functions. This formula facilitates the application of the Barvinok algorithm in the primal space for cone decomposition. Moreover, by considering the decomposition of the cone and its dual cone from the side with a smaller index, we design a new algorithm named PDBarv. This algorithm integrates the primal Barvinok algorithm with the dual Barvinok algorithm, resulting in appreciable performance improvements.

The remainder of this paper is organized as follows. Section 2 provides a brief introduction to lattice-point generating functions of polyhedra and the Barvinok algorithm. Section 3 details the relationship between half-open cones and closed cones, presenting the closed-form formula for the decomposition of simplicial cones, supported by proofs from two distinct perspectives. Section 4 describes our algorithm, which is based on the closed-form formula from Section 3 and the strategy of decomposing cones with the smaller index, followed by experimental comparisons. Finally, Section 5 concludes the paper with a discussion of the implications and future directions for this research.

\section{Generating Functions and the Barvinok Algorithm}
In this section, we provide a brief overview of the fundamental concepts related to lattice-point generating functions and the Barvinok algorithm. These concepts serve as the cornerstone for the improvements and the new algorithm that we propose.

\subsection{Generating Functions}

Let \( S \subset \mathbb{R}^{d} \) be a set. We define its \emph{lattice-point generating function} (or \emph{integer-point transform}) as the formal Laurent series
\[
\sigma(S; \mathbf{y}) = \sigma(S; y_{1}, \ldots, y_{d}) := \sum_{\mathbf{m} \in S \cap \mathbb{Z}^{d}} \mathbf{y}^{\mathbf{m}},
\]
where \(\mathbf{y}^{\mathbf{m}} = \prod_{i=1}^d y_i^{m_i}\). When the variables are unambiguous, we abbreviate it as \(\sigma(S)\). This function encodes all lattice points within \( S \), and under suitable convergence conditions, it admits a rational function representation (see Lemma~11.1 in \cite{Beck-book-computing-Discre}).

Throughout this paper, we restrict our attention to \emph{rational polyhedra}, and unless stated otherwise, all polyhedra are assumed to be rational. The enumeration of lattice points in polyhedra frequently reduces to the study of cones. A foundational result of their generating functions is the following:

\begin{lemma}~\cite{Dual-barvinok} \label{lem-line}
If \( P \subset \mathbb{R}^n \) is a rational cone containing a line, then \(\sigma(P) \equiv 0\).
\end{lemma}

For simplicial cones, Ehrhart's formula~\cite{ehrhart1977} provides an explicit representation:

\begin{theorem}
Let \( K := \bigl\{ \sum_{i=1}^d k_i \alpha_i \mid k_i \geq 0 \bigr\} \) be a full-dimensional simplicial cone in \(\mathbb{R}^d\), where \(\alpha_1, \ldots, \alpha_d \in \mathbb{Z}^d\) are primitive vectors, i.e., with greatest common divisor $1$. Then,
\begin{equation}
\sigma(K; \mathbf{y}) = \frac{\sigma(\Pi; \mathbf{y})}{\prod_{i=1}^d (1 - \mathbf{y}^{\alpha_i})},
\end{equation}
where \(\Pi := \bigl\{ \sum_{i=1}^d k_i \alpha_i \mid 0 \leq k_i < 1 \bigr\}\) denotes the \emph{half-open fundamental parallelepiped} of \( K \).
\end{theorem}

The cardinality of \( \Pi \cap \mathbb{Z}^d \), called the \emph{index} of \( K \) (denoted \(\mathrm{ind}(K)\)), determines whether \( K \) is unimodular (i.e., \(\mathrm{ind}(K) = 1\)).

A pivotal result in polyhedral combinatorics is Brion's theorem~\cite{Brion-Thm-and-dual-trick}, which decomposes the generating function of a polyhedron into contributions from its tangent cones:

\begin{theorem}
Let \( P \subset \mathbb{R}^n \) be a rational polyhedron. Then,
\begin{equation}
\sigma(P; \mathbf{y}) = \sum_{v \in \mathrm{Vert}(P)} \sigma(K_v; \mathbf{y}),
\end{equation}
where \( K_v := \{ v + t(z - v) \mid z \in P, \, t \geq 0 \} \) is the \emph{tangent cone} of \( P \) at vertex \( v \).
\end{theorem}

\begin{example}
Consider the interval \( L := [0, 3] \subset \mathbb{R} \). Applying Brion's theorem to its vertices \( \{0, 3\} \) yields:
\begin{equation}
\sigma(L; y) = \frac{1}{1-y} + \frac{y^3}{1-y^{-1}} = 1 + y + y^2 + y^3,
\end{equation}
which matches the direct enumeration of lattice points \( \{0, 1, 2, 3\} \).
\end{example}

\subsection{The Barvinok Algorithm}
The original Barvinok algorithm is a fundamental tool for computing the generating functions of lattice points in rational polyhedra. It operates in primal space and recursively decomposes cones using the following lemma, which is implicitly stated in \cite{Primal-barvinok}.

\begin{lemma}\label{l-cone-dec}
Let $K$ be the cone generated by a nonsingular matrix $A \in \mathbb{Z}^{d \times d}$, and let $[\cdot ]$ denote the indicator function.
If $w = A\alpha$ is an integral vector, where $\alpha = (k_1, \dots, k_d)^T$ has at least one positive entry, then
\begin{equation}\label{dual_bar_formula}
  [\cone(A)] \equiv \sum_{i=1}^{d} \mathrm{sgn}(k_i) [\cone(A_i)] \mod \text{lower-dimensional cones},
\end{equation}
where $A_i$ is the matrix obtained by replacing the $i$-th column of $A$ with $w$.
\end{lemma}

In 1994, Barvinok establishes the existence of a \emph{good integer vector} \( w \) with each $|k_i| \leq \ind(K)^{-1/d}$ using the \emph{Minkowski Convex Body Theorem}.
Repeated application of Lemma~\ref{l-cone-dec} using good $w$ then yields the decomposition formula:
\begin{equation} \label{sign_sum_smallcones}
  \sigma(K;\mathbf{y}) = \sum_{i \in I} \epsilon_{i} \sigma(\cone(B_i)) = \sum_{i \in I} \epsilon_{i} \frac{1}{\prod_{j=1}^{d} \left(1 - \mathbf{y}^{\mathbf{b}_{ij}} \right)},
\end{equation}
where each \( \cone(B_i) \) is unimodular, \( \mathbf{b}_{ij} \in \mathbb{Z}^{d} \), \( \epsilon_{i} = \pm 1 \), and the size $|I|$ is bounded by a polynomial in $\log (\ind(K))$.

This process is referred to as the \emph{Original Primal Barvinok Algorithm}. Note that we need to record the proper faces of the intersecting cones.

Due to the complexity of tracking all proper faces of intersecting cones, Barvinok and Pommersheim introduced the \emph{Dual-Barvinok algorithm} \cite{Dual-barvinok}. The primary advancement of this algorithm lies in performing decomposition in dual space: it discards lower-dimensional cones and fully decomposes into unimodular cones before dualizing back.

Building on the above ideas, De Loera et. al. in 2004 utilized the  LLL (Lenstra–Lenstra–Lovász) lattice basis reduction algorithm~\cite{LLL1982} to approximate \( w \), enhancing the practical implementation of Barvinok's method \cite{De-loear-LLLwork-finish-Bar}. Specifically, writing \( w = A\alpha \) as $\alpha = A^{-1}w$. Then $\alpha$ lies in the lattice $L(A^{-1})$ spanned by the columns of $A^{-1}$, and we seek a small vector under the infinite norm $\|\cdot\|_\infty$. We approximate $\alpha$ using the LLL algorithm as follows: first apply LLL to $A^{-1}$ to obtain a reduced basis $\delta_1,\dots, \delta_d$. Each $\delta_i$ then provides a good approximation to the shortest vector in $L(A^{-1})$ under the Euclidean norm $\|\cdot\|_2$. This $\delta_i$ turns out to be also good approximations under the $\|\cdot\|_\infty$-norm. Therefore, we can approximate $\alpha$ by identifying the smallest vector among the $\delta_i$'s. This approach provided a practical implementation that served as the foundation for the \emph{LattE} software package.

Since the Dual-Barvinok algorithm becomes inefficient when $\ind(K)$ is small while $\ind(K^*)$ is large,
Matthias K\"oppe developed the \emph{irrational perturbation method} to address limitations in the original Barvinok algorithm \cite{Primal-irrational}.
This technique constructs a cone's \emph{stability cube} and selects a perturbation vector
to translate cones in primal space, ensuring no integer points lie on the proper faces of intersecting cones.
Although innovative, it encountered difficulties in constructing such  vectors, lacked reusability, and could not handle counting problems for parametric polyhedra.

To overcome these limitations, K\"oppe and Verdoolaege developed a half-open cone decomposition method \cite{Half-open-dec},
establishing a classification model that enhances the computation of lattice-point generating functions for parametric polytopes.

In fact, from the generating function perspective,  \emph{half-open cones} and \emph{closed cones} are mutually convertible for simplicial cones. In the next section, we present a novel approach to decomposing a simplicial cone into a signed sum of simplicial cones.

\section{A Unified Simplicial Cone Decomposition Formula}
We introduce a novel decomposition formula for simplicial cones that incorporates an additional vector $\gamma$. This formula decomposes the generating function of a simplicial cone into a signed sum of generating functions for other simplicial cones, forming the foundation for our proposed PDBarv algorithm. We provide two distinct proofs: one geometric and one algebraic.

\subsection{The Decomposition Formula}
Let $A = (\alpha_1, \alpha_2, \ldots, \alpha_n)$ be a matrix in $\mathbb{R}^{d \times n}$ of rank $n$. For $v \in \mathbb{R}^d$ and $\theta \subset [n]$, define
\begin{equation}\label{eq-defK}
    K = \cone^{\theta}(A; v) := v + \left\{ \sum_{i=1}^n k_i \alpha_i \;:\; k_i \geq 0 \text{ for } i \notin \theta \text{ and } k_i > 0 \text{ for } i \in \theta \right\}.
\end{equation}
We call $K$ a half-open simplicial cone shifted at the vertex $v$, with open facets indexed by the set $\theta$. When $\theta$ is empty, we omit the superscript $\theta$, and when $v = \mathbf{0}$, we omit $v$.

For a vector $\gamma \in \mathbb{R}^d$, denote by $A[(i \to \gamma)]$ the matrix obtained by replacing the $i$-th column of $A$ with $\gamma$. In particular, $A[(i \to l\alpha_i)]$ is abbreviated as $A[(i \to l)]$ for $l \in \mathbb{Q}$, and $A[(i\to 0)] = A[(i \to \mathbf{0})]$ denotes the matrix with the $i$-th column removed. More generally, $A[(i_1, \ldots, i_k) \to (\gamma_1, \ldots, \gamma_k)]$ is defined analogously, while $A[(i_1, \ldots, i_k) \to l]$ denotes the matrix with columns $i_1,\ldots,i_k$ multiplied by $l$.

We now present the main theorem, which constitutes the core contribution of our work and establishes the foundation for subsequent developments.
\begin{theorem}\label{core_thm}
Let $K = \cone(\alpha_1, \dots, \alpha_n)$ be a simplicial cone. Given a vector $\gamma=p_{1}\alpha_{1}+\cdots+p_{k}\alpha_{k}-p_{k+1}\alpha_{k+1}-\dots-p_{r}\alpha_{r}$,
where $p_i \in \mathbb{Q}_{>0}$ for all $i$ and $k \in \{0,1,\ldots,r\}$, we have a decomposition of $\sigma(K)$ into closed simplicial cones
 as follows:
\begin{equation}\label{core-closed-equation}
\sigma(K) = \sum_{i=1}^{k} \sigma(\cone(A_i)) (-1)^{i-1} + \sum_{j=k+1}^{r} \sigma(\cone(A_j)) (-1)^{j-(k+1)},
\end{equation}
where $A_i = A[(1,\ldots,i) \to (-1,\ldots,-1,\gamma)]$ and $A_j = A[(k+1,\ldots,j) \to (-1,\ldots,-1,-\gamma)]$.
\end{theorem}

\begin{remark}
To work in the dual space, \eqref{core-closed-equation} requires modification such that only generators $\alpha_i$ and $\gamma$ appear. The resulting expression then matches that in Lemma \ref{l-cone-dec}, with $\sigma$ replaced by the indicator function. Without this adjustment, dualizing \eqref{core-closed-equation} only yields an equality when modulo lower-dimensional cones.
\end{remark}

\subsection{Geometric Proof}
We first establish a fundamental relationship between half-open and closed cones, which is pivotal to our cone decomposition methodology.
\begin{lemma}\label{half_and_close_change}
From the generating function view, half-open cones and closed cones are mutually convertible. Precisely,
let $A = (\alpha_1, \alpha_2, \ldots, \alpha_n)$ generate the rational simplicial cone $\cone(A)$. Then the following identities hold:
\begin{enumerate}
    \item[\textnormal{(i)}] $\sigma(\cone^{\{i\}}(A)) = -\sigma(\cone(A[(i \to -1)]))$,
    \item[\textnormal{(ii)}] $\sigma(\cone^{\theta}(A)) = (-1)^{|\theta|} \sigma(\cone(A[(\theta \to -1)]))$,
    \item[\textnormal{(iii)}] $\sigma(\cone(A)) + \sigma(\cone(A[(i \to -1)])) = \sigma(\cone(A[(i \to \mathbf{0})]))$.
\end{enumerate}
\end{lemma}

\begin{proof}[Proof]
For part (i), by the definition of a half-open cone, we have
\[
\sigma(\cone(A)) + \sigma(\cone^{\{i\}}(A[(i \to -1)])) = \sigma(K^{\prime}),
\]
where
\[
K^{\prime} = \left\{ \sum_{j=1}^n k_j \alpha_j : k_j \geq 0 \text{ for } j \neq i, \text{ and } k_i \in \mathbb{R} \right\}.
\]
This is a cone containing the line generated by $\alpha_i$. The first part then follows from Lemma \ref{lem-line}.

Part (ii) follows by repeatedly applying part (i).

For part (iii), we use part (i) and the observation that
\[
\sigma(\cone(A[(i \to -1)])) = \sigma(\cone^{\{i\}}(A[(i \to -1)])) + \sigma(\cone(A[(i \to \mathbf{0})])).
\]
This completes the proof.
\end{proof}

Lemma~\ref{half_and_close_change} allows us to translate Theorem~\ref{core_thm} in terms of half-open cones.
\begin{proposition}\label{core_pro}
Let \(K\) and \(\gamma\) be as in Theorem \ref{core_thm}, then we have
\begin{equation}\label{core_formula_half}
\sigma(K)=\sum_{i=1}^{k}\sigma(\cone^{[i-1]}(A[(i\to\gamma)])))+\sum_{j=k+1}^{r}\sigma(\cone^{\{k+1,\ldots,j-1\}}(A[(j\to-\gamma)])).
\end{equation}
\end{proposition}

Before giving the proof, let us consider the case when $\gamma$ is an interior point. The decomposition in~\eqref{e-interior} below is likely known, but we provide a proof in Appendix \ref{appendix1} for completeness.

\begin{proposition}\label{int_Dec}
Let $A = (\alpha_1, \alpha_2, \ldots, \alpha_n)$ generate an $n$-dimensional rational simplicial cone $K = \cone(A)$ in $\mathbb{R}^d$. For any $\gamma = \sum_{i=1}^r k_i\alpha_i \in K$ with $k_i > 0$,
\begin{equation}\label{e-interior}
    \sigma(K) = \sum_{i=1}^r \sigma(K_i),
\end{equation}
where $K_i = \cone^{[i-1]}\left(A[(i \to \gamma)]\right)$.
\end{proposition}

\begin{example}\label{ex:3d-decomp}
The three-dimensional case provides a visual illustration. Consider $K = \cone(\alpha_1, \alpha_2, \alpha_3)$ and $\gamma = \alpha_1 + \alpha_2 + \alpha_3$. The cross-section shows:
\begin{itemize}
    \item Three subcones $K_1, K_2, K_3$ partitioning $K$
    \item Dashed boundaries indicating open facets in each $K_i$
    \item The central point $\gamma$ serving as the decomposition pivot
\end{itemize}

\begin{figure}[H]
    \centering
    \includegraphics[width=0.5\linewidth]{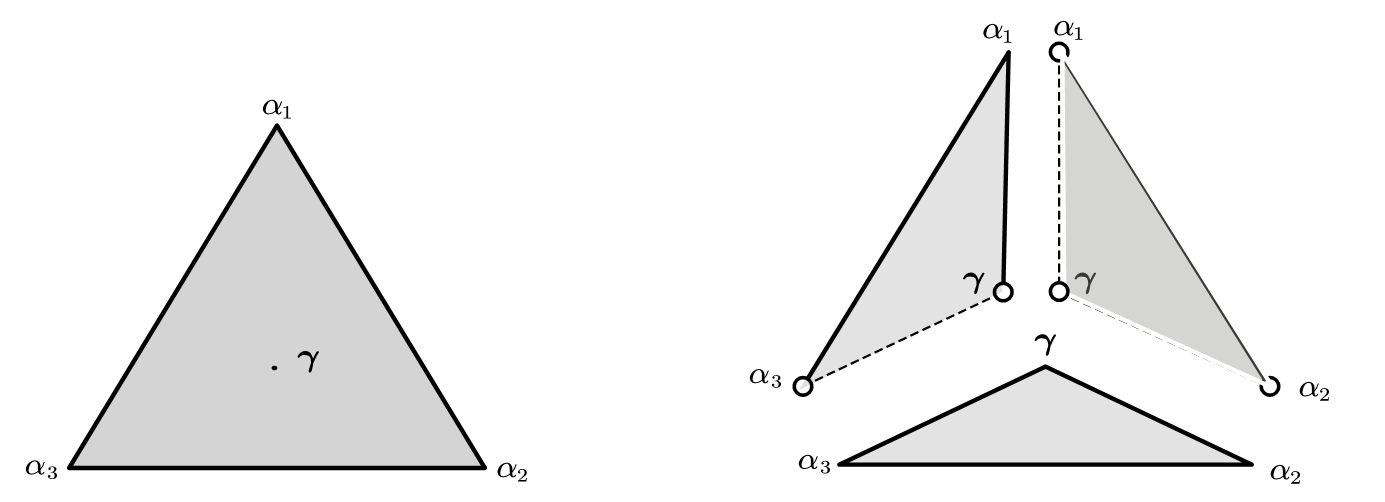}
    \caption{Cross-section of 3D cone decomposition. Subcones $K_i$ are colored distinctly, with dashed edges denoting open facets.}
    \label{fig:3d-decomp}
\end{figure}
\end{example}

We require an additional decomposition involving lower-dimensional cones, proved in Appendix \ref{appendix2}.
\begin{lemma}\label{thm_lowd_dec}
Let $A = (\alpha_1, \ldots, \alpha_n)$ generate an $n$-dimensional rational simplicial cone $K = \cone(A)$. Then
\begin{equation}
    \sigma(K) = \sum_{i=1}^{r-1} \sigma(F_i) + \sigma\left(\cone^{[r-1]}(A)\right),
\end{equation}
where $F_i = \cone^{[i-1]}\big(A[(i \to 0)]\big)$ for $1 \leq i \leq r-1$.
\end{lemma}

We now prove Proposition \ref{core_pro}.
\begin{proof}[Proof of Proposition \ref{core_pro}]
Without loss of generality, we may assume that $\gamma$ is of the form $\alpha_{1}+\cdots+\alpha_{k}-\alpha_{k+1}-\dots-\alpha_{r}$. Let $m$ count the negative coefficients in $\gamma$. We proceed by induction on $m$.

The base case $m = 0$ follows from Proposition \ref{int_Dec}.

Assume the result holds for $m = r - k$. For $m=r-k+1$ ($\gamma=\alpha_{1}+\cdots+\alpha_{k-1}-\alpha_{k}-\dots-\alpha_{r}$), we must show
 \begin{equation}\label{e-k-gamma}
\sigma(K)=\sum_{i=1}^{k-1}\sigma(\cone^{[i-1]}(A[(i\to\gamma)])))+\sum_{j=k}^{r}\sigma(\cone^{\{k,\ldots,j-1\}}(A[(j\to-\gamma)])).
\end{equation}

Apply the induction hypothesis to $\hat{K}=\cone(\hat{A})$
with respect to $\gamma =\alpha_{1}+\cdots+(-\alpha_{k})-\alpha_{k+1}-\dots-\alpha_{r}$, where $\hat{A}=(\alpha _{1},\ldots \alpha _{k-1},-\alpha _{k},\alpha_{k+1},\dots,\alpha _{n})$. We obtain
\begin{equation} \label{e-k1-gamma}
\sigma(\hat K)=\sum_{i=1}^{k}\sigma(\cone^{[i-1]}(\hat{A}[(i\to\gamma)]))+\sum_{j=k+1}^{r}\sigma(\cone^{\{k+1,\ldots,j-1\}}(\hat{A}[(j\to-\gamma)])).
\end{equation}

It suffices to show the truth of equality obtained by taking the sum of the two equations~\eqref{e-k-gamma} and \eqref{e-k1-gamma}.
By Lemma \ref{half_and_close_change}, we have
\begin{align*}
  &\sigma(K)+ \sigma(\hat{K})= \sigma(\cone\left(\alpha _{1},\ldots \alpha _{k-1},\alpha_{k+1},\dots,\alpha _{n}\right)), \\
  &\sigma(\cone^{\{k,\ldots,j-1\}}(A[(j\to-\gamma)]))+\sigma(\cone^{\{k+1,\ldots,j-1\}}(\hat{A}[(j\to-\gamma)]))=0, \text{ when }j>k,\\
  &\sigma(\cone^{[i-1]}(A[(i\to\gamma)]))+\sigma(\cone^{[i-1]}(\hat{A}[(i\to\gamma)]))=\sigma(\cone^{[i-1]}(A[(k,i)\to(\mathbf{0},\gamma)])), \text{ when }i<k.
\end{align*}
The first equation gives the left-hand side, and the other two are used to simplify the right-hand side. We obtain
$$\sum_{j=k}^{r}\sigma(\cone^{\{k,\ldots,j-1\}}(A[(j\to-\gamma)]))+
\sum_{j=k+1}^{r}\sigma(\cone^{\{k+1,\ldots,j-1\}}(\hat{A}[(j\to-\gamma)]))=\sigma(\cone(A[(k\to-\gamma)])),$$
and similarly,
\begin{align*}
   & \sum_{i=1}^{k-1}\sigma(\cone^{[i-1]}(A[(i\to\gamma)]))+\sum_{i=1}^{k}\sigma(\cone^{[i-1]}(\hat{A}[(i\to\gamma)])) \\
  = & \sum_{i=1}^{k-1}\sigma(\cone^{[i-1]}(A[(k,i)\to(\mathbf{0},\gamma)]))+\sigma(\cone^{[k-1]}(A[(k\to\gamma)])) \\
  =&\sigma(\cone(A[(k\to\gamma)])).  \text{ \quad    (by Lemma \ref{thm_lowd_dec})    }
\end{align*}
Since, by Lemma~\ref{half_and_close_change} again,
\begin{equation}\label{RHS}
\sigma(\cone(A[(k\to\gamma)]))+\sigma(\cone(A[(k\to-\gamma)]))=\sigma(\cone\left( \alpha _{1},\ldots \alpha _{k-1},\alpha_{k+1},\dots,\alpha _{n}\right)),
\end{equation}
the proposition holds.
\end{proof}

\begin{example}
We demonstrate the decomposition for $\sigma(\cone(\alpha_{1}, \alpha_{2}))$ in $\mathbb{R}^2$ using carefully chosen points $\gamma_{1},\gamma_{2},\gamma_{3},\gamma_{4}$. Their geometric configuration relative to the cone generators is shown in Figure~\ref{fig:gamma_config}.
\end{example}

\begin{figure}[h]
    \centering
    \includegraphics[width=0.3\linewidth]{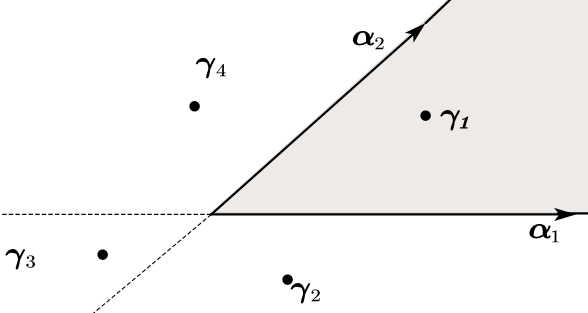}
    \caption{Positions of points $\gamma_i$ relative to $\cone(\alpha_1,\alpha_2)$}
    \label{fig:gamma_config}
\end{figure}

Employing $\gamma_{1}, \gamma_{2}, \gamma_{3}$, and $\gamma_{4}$, the generating function decompositions given by Proposition~\ref{core_pro} are:
\begin{itemize}
    \item Point $\gamma_{1}$: $\sigma(\cone(\alpha_{2}, \gamma_{1})) - \sigma(\cone(-\alpha_{1}, \gamma_{1}))$,
    \item Point $\gamma_{2}$: $\sigma(\cone(\alpha_{2}, \gamma_{2})) + \sigma(\cone(\alpha_{1}, -\gamma_{2}))$,
    \item Point $\gamma_{3}$: $\sigma(\cone(\alpha_{2}, -\gamma_{3})) - \sigma(\cone(-\alpha_{1}, -\gamma_{3}))$,
    \item Point $\gamma_{4}$: $\sigma(\cone(\alpha_{1}, \gamma_{4})) + \sigma(\cone(\alpha_{2}, -\gamma_{4}))$.
\end{itemize}
The corresponding decomposition diagrams are shown below.

\begin{figure}[H]
    \centering
    \subfloat[Decomposition by  $\gamma_{1}$ ]{\includegraphics[width=4cm,height=3cm]{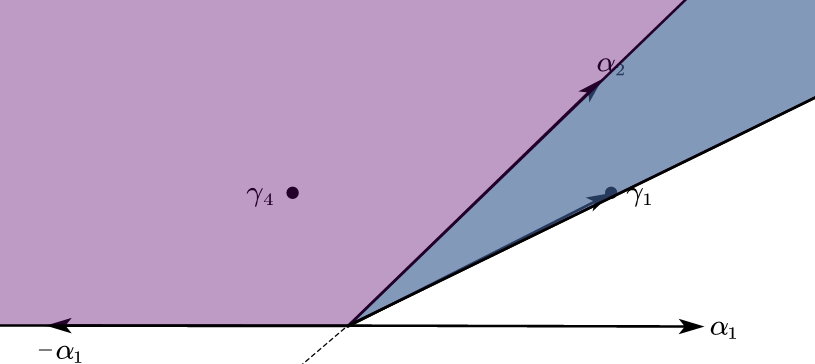}}\xspace
    \subfloat[Decomposition by  $\gamma_{2}$ ]{\includegraphics[width=4cm,height=3cm]{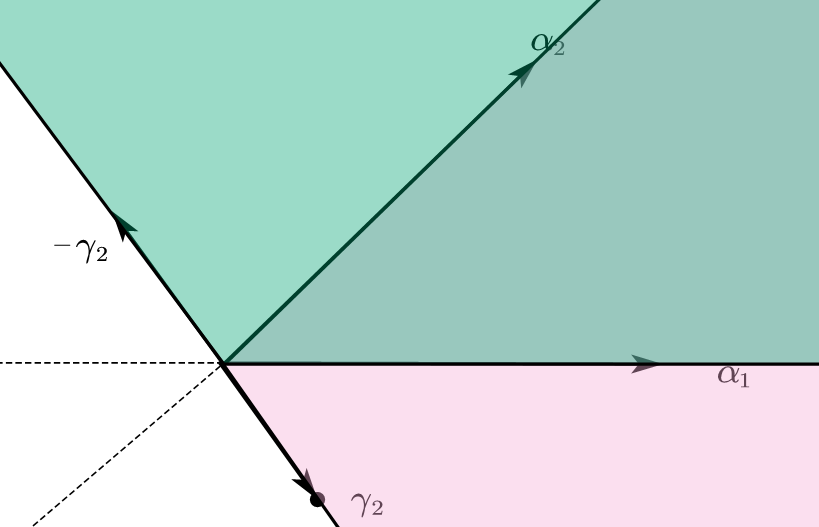}}\xspace

    \subfloat[Decomposition by  $\gamma_{3}$ ]{\includegraphics[width=4cm,height=3cm]{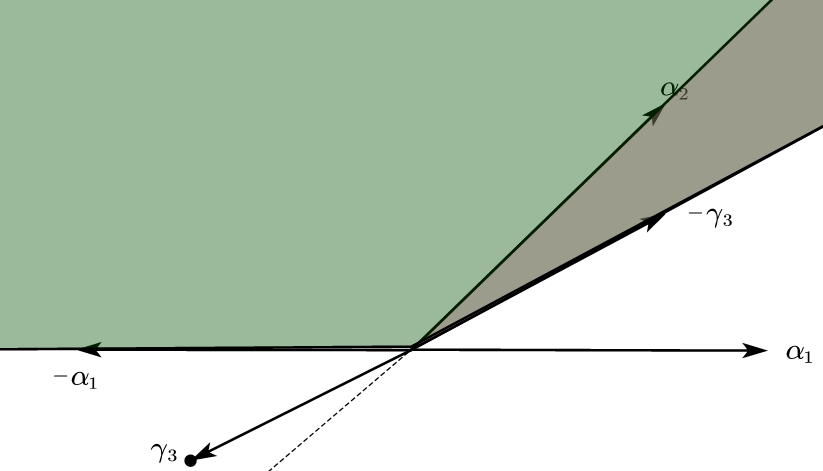}}\xspace
    \subfloat[Decomposition by  $\gamma_{4}$ ]{\includegraphics[width=4cm,height=3cm]{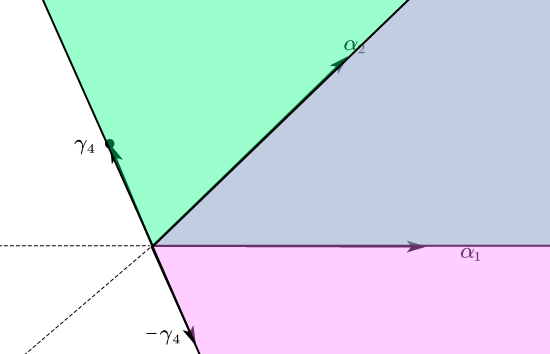}}\xspace
    \caption{Decomposition diagrams}
    \label{fig:decomposition}
\end{figure}

We conclude this subsection with a proof of Theorem~\ref{core_thm}.
\begin{proof}[Geometric Proof of Theorem \ref{core_thm}]
The expression in Theorem~\ref{core_thm} is equivalent to that in Proposition~\ref{core_pro} by the equalities
\begin{align*}
\sigma(\cone^{[i-1]}(A[(i\to\gamma)])) &= \sigma(\cone(A_i))(-1)^{i-1}, \\
\sigma(\cone^{\{k+1,\ldots,j-1\}}(A[(j\to-\gamma)])) &= \sigma(\cone(A_j))(-1)^{j-(k+1)},
\end{align*}
which are derived from Lemma~\ref{half_and_close_change}.
\end{proof}

\subsection{Algebraic Proof}
We present an alternative algebraic proof of Theorem~\ref{core_thm}. Central to this approach is the linear operator $\mathcal{Z}_{y_{i}}$ defined in~\cite{simplecone}:
\begin{equation}\label{Zyi-def}
  \mathcal{Z}_{y_{j}} \sum_{\kappa \in \mathbb{Q}^{n}} a_{\kappa}\mathbf{y}^{\kappa} := \sum_{\substack{\kappa \in \mathbb{Q}^{n} \\ \kappa_{j} \in \mathbb{Z}}} a_{\kappa} \mathbf{y}^{\kappa},
\end{equation}
where the coefficients $a_{\kappa}$ are independent of the $y$ variables. This operator extends component-wise to a multivariate version $\Zy$.

A key property from~\cite[Proposition 9]{simplecone} ensures $\Zy$ preserves rationality: it maps rational functions to rational functions, and the result is independent of the series expansion method.

Systematically developing this approach requires a lattice structure result concerning the column space $\operatorname{Col}(A)$.
\begin{lemma}\label{col(A)}
For any nonsingular rational $d\times d$ matrix $A = (\alpha_{1},\ldots,\alpha_{d})$, there exists $M \in \mathbb{Z}_{>0}$ such that
\begin{equation}
  \operatorname{Col}(A) \cap \mathbb{Z}^d = \left\{ \sum_{i=1}^{d} k_i\frac{\alpha_i}{M} : k_i \in \mathbb{Z} \right\} \cap \mathbb{Z}^d.
\end{equation}
Equivalently, every lattice point in $\operatorname{Col}(A)$ has an $A$-basis representation with coordinates whose denominators divide $M$.
\end{lemma}

\begin{proof}
Let $m \in \mathbb{Z}_{>0}$ satisfy $mA \in \mathbb{Z}^{d \times d}$, which exists by the rationality of $A$. Consider any lattice point $\beta \in \operatorname{Col}(A) \cap \mathbb{Z}^d$ expressed in the $A$-basis as
\[
\beta = \sum_{i=1}^{d} q_i \frac{\alpha_i}{M}
\]
for some $M \in \mathbb{Z}_{>0}$ to be determined. Applying Cramer's Rule to the scaled system $mA\vec{x} = m\beta$ gives the coordinates:
\[
q_i = \frac{M}{\det(mA)} \det\left( m A[(i \to \beta)] \right),
\]
where recall that $A[(i \to \beta)]$ denotes $A$ with its $i$-th column replaced by $\beta$. Choosing $M$ as any positive integer multiple of $|\det(mA)|$ ensures $q_i \in \mathbb{Z}$ for all $i$. Consequently, $m q_i \in \mathbb{Z}$.
\end{proof}

From the perspective of $\Zy$, we obtain an alternative characterization of generating functions for half-open cones and their relation to closed cones, captured in the following proposition.

\begin{proposition}\label{thm_no_PI}
Let $K = \cone(A) = \cone(\alpha_1,\ldots,\alpha_n)$ be a rational simplicial cone in $\mathbb{R}^d$ with generators $\alpha_i \in \mathbb{Q}^d$. There exists $M \in \mathbb{Z}_{>0}$ such that for any half-open cone $K^\theta = \cone^\theta(\alpha_1,\ldots,\alpha_n)$, the generating function satisfies
\begin{align}
\sigma(K^{\theta};\y) = \Zy \left( \prod_{i \in \theta} \frac{\mathbf{y}^{\alpha_i/M}}{1 - \mathbf{y}^{\alpha_i/M}} \prod_{j \notin \theta} \frac{1}{1 - \mathbf{y}^{\alpha_j/M}} \right).
\end{align}
Consequently, we have the reciprocity relation
\begin{align}
\sigma(\cone^{\theta}(A);\y) = (-1)^{|\theta|} \sigma(\cone(A[(\theta \to -1)])).
\end{align}
Thus, any half-open cone can be represented as a closed cone via coordinate reflection.
\end{proposition}

\begin{proof}[Proof]
To establish the first equation, we aim to find a positive integer $M \in \mathbb{Z}_{>0}$ such that
\begin{align*}
K^\theta \cap \mathbb{Z}^{d} = \left\{ q_1 \frac{\alpha_1}{M} + \cdots + q_n \frac{\alpha_n}{M} \in \mathbb{Z}^d \,\bigg|\, q_i \in \mathbb{N} \text{ for } i \notin \theta, \, q_j \in \mathbb{Z}_{>0} \text{ for } j \in \theta \right\}.
\end{align*}
The right-hand side (RHS) is contained within the left-hand side (LHS) for any $M>0$. For the reverse inclusion, consider any $\beta = q_1 \frac{\alpha_1}{M} + \cdots + q_n \frac{\alpha_n}{M} \in K^\theta \cap \mathbb{Z}^d$, where $M$ is to be determined. By definition, $q_i \geq 0$ for $i \notin \theta$ and $q_j > 0$ for $j \in \theta$.

Since $K$ is simplicial, we can assume without loss of generality that the first $n$ rows and columns of $A$ form a nonsingular square matrix $A' = (\alpha_1', \ldots, \alpha_n')$. Let $\beta'$ denote the vector formed by the first $n$ coordinates of $\beta$. The system $A \vec{x} = \beta$ is equivalent to $A' \vec{x} = \beta'$, and the existence of such an $M$ is guaranteed by Lemma \ref{col(A)}.

For the second equation, apply the same $M$ to the cone $\cone(A[(\theta \to -1)])$. The reciprocity follows since for each $i \in \theta$:
\[
\frac{\mathbf{y}^{\alpha_i/M}}{1 - \mathbf{y}^{\alpha_i/M}} = \frac{-1}{1 - \mathbf{y}^{-\alpha_i/M}}.
\]
\end{proof}

By utilizing the properties of the $\Zy$ operator,  we can also derive Equation \eqref{core-closed-equation}. To this end, the following lemma is essential.
\begin{lemma}\label{dec_rational_function}
  For $\gamma = \alpha_1 + \cdots + \alpha_k - \alpha_{k+1} - \cdots - \alpha_r \in \mathbb{R}^d$ with $r \leq n$, we have
  \begin{equation}
    \dfrac{1}{\prod_{i=1}^{r} (1 - \mathbf{y}^{\alpha_i})} = L_1(\mathbf{y}) + L_2(\mathbf{y}),
  \end{equation}
  where
  \begin{align}
    L_1(\mathbf{y}) &= \sum_{j=1}^{k} \frac{(-1)^{j-1}}{(1 - \mathbf{y}^\gamma) \prod_{i=1}^{j-1} (1 - \mathbf{y}^{-\alpha_i}) \prod_{i=j+1}^{r} (1 - \mathbf{y}^{\alpha_i})}, \\
    L_2(\mathbf{y}) &= \sum_{j=k+1}^{r} \frac{(-1)^{j-(k+1)}}{(1 - \mathbf{y}^{-\gamma}) \prod_{i=1}^{k} (1 - \mathbf{y}^{\alpha_i}) \prod_{i=k+1}^{j-1} (1 - \mathbf{y}^{-\alpha_i}) \prod_{i=j+1}^{r} (1 - \mathbf{y}^{\alpha_i})}.
  \end{align}
\end{lemma}
\begin{proof}[Proof]
It is clear that
\begin{align*}
  \dfrac{1}{\prod\limits_{i=1}^{r}(1-\mathbf{y}^{\alpha_{i}})}= & \dfrac{1-\mathbf{y}^{\alpha_{1}+\cdots+\alpha_{k}}-\mathbf{y}^{\gamma}(1-\mathbf{y}^{\alpha_{k+1}+\cdots+\alpha_{r}})}{(1-\mathbf{y}^\gamma)\prod\limits_{i=1}^{r}(1-\mathbf{y}^{\alpha_{i}})} \\
  = & \dfrac{1-\mathbf{y}^{\alpha_{1}+\cdots+\alpha_{k}}}{(1-\mathbf{y}^\gamma)\prod\limits_{i=1}^{r}(1-\mathbf{y}^{\alpha_{i}})}+
\dfrac{(1-\mathbf{y}^{\alpha_{k+1}+\cdots+\alpha_{r}})}{(1-\mathbf{y}^{-\gamma})\prod\limits_{i=1}^{r}(1-\mathbf{y}^{\alpha_{i}})}.
\end{align*}
Elementary algebra gives
\begin{align*}
 \dfrac{1-\mathbf{y}^{\alpha_{1}+\alpha_{2}+\cdots+\alpha_{k}}}{(1-\mathbf{y}^\gamma)\prod\limits_{i=1}^{r}(1-\mathbf{y}^{\alpha_{i}})}  = & \dfrac{\sum_{j=1}^{k}\mathbf{y}^{\alpha_1\cdots\alpha_{j-1}}(1-\mathbf{y}^{\alpha_j})}{(1-\mathbf{y}^\gamma)\prod\limits_{i=1}^{r}(1-\mathbf{y}^{\alpha_{i}})} \\
  = & \sum_{j=1}^{k} \frac{(-1)^{j-1}}{(1-\mathbf{y}^\gamma)\prod\limits_{i=1}^{j-1}(1-\mathbf{y}^{-\alpha_{i}})\prod\limits_{i=j+1}^{r}(1-\mathbf{y}^{\alpha_{i}})}.
\end{align*}
Similarly, we have
$$\dfrac{(1-\mathbf{y}^{\alpha_{k+1}+\cdots+\alpha_{r}})}{(1-\mathbf{y}^{-\gamma})\prod\limits_{i=1}^{r}(1-\mathbf{y}^{\alpha_{i}})}= \sum\limits_{j=k+1}^{r}\frac{(-1)^{j-(k+1)}}{(1-\mathbf{y}^{-\gamma})\prod\limits_{i=1}^{k}(1-\mathbf{y}^{\alpha_i})
\prod\limits_{i=k+1}^{j-1}(1-\mathbf{y}^{-\alpha_{i}})\prod\limits_{i=j+1}^{r}(1-\mathbf{y}^{\alpha_{i}})}.$$
\end{proof}

We now provide a new proof of Theorem~\ref{core_thm} via the $\Zy$ operator.
\begin{proof}[Algebraic Proof of Theorem \ref{core_thm}]
  It suffices to verify Equation~\eqref{core-closed-equation} for $\gamma = \alpha_1 + \cdots + \alpha_k - \alpha_{k+1} - \cdots - \alpha_r$.

  By Theorem~\ref{thm_no_PI}, there exists $M \in \mathbb{Z}_{>0}$ such that
  \begin{equation*}
    \sigma(K) = \Zy \frac{1}{\left(1 - \mathbf{y}^{\frac{\alpha_1}{M}}\right) \left(1 - \mathbf{y}^{\frac{\alpha_2}{M}}\right) \cdots \left(1 - \mathbf{y}^{\frac{\alpha_n}{M}}\right)}.
  \end{equation*}
  An analogous formula holds for each cone on the right-hand side of \eqref{core-closed-equation}, though with a distinct $M_i$ for the $i$-th cone. Replacing $M$ by a common multiple of $M, M_1, \dots, M_r$ ensures all formulas use identical $M$.

 By the linearity of $\Zy$ and Lemma~\ref{dec_rational_function},
  \begin{equation*}
    \sigma(K) = \Zy \frac{L_1(\mathbf{y}^{\frac{1}{M}})}{\prod_{i=r+1}^{n} \left(1 - \mathbf{y}^{\frac{\alpha_i}{M}}\right)} + \Zy \frac{L_2(\mathbf{y}^{\frac{1}{M}})}{\prod_{i=r+1}^{n} \left(1 - \mathbf{y}^{\frac{\alpha_i}{M}}\right)}.
  \end{equation*}
  Furthermore,

  \begin{align}
     & \Zy \frac{ L_1(\mathbf{y}^{\frac{1}{M}})}{\prod\limits_{i=r+1}^{n}\left(1-\mathbf{y}^{\frac{\alpha_i}{M}}\right)}=\sum\limits_{j=1}^{k}\Zy\frac{(-1)^{j-1}}{(1-\mathbf{y}^\frac{\gamma}{M})\prod\limits_{i=1}^{j-1}(1-\mathbf{y}^{\frac{-\alpha_{i}}{M} })
  \prod\limits_{i=j+1}^{n}(1-\mathbf{y}^{\frac{\alpha_{i}}{M} })}, \\
     & \Zy \frac{ L_2(\mathbf{y}^{\frac{1}{M}})}{\prod\limits_{i=r+1}^{n}\left(1-\mathbf{y}^{\frac{\alpha_i}{M}}\right)}=\sum\limits_{j=k+1}^{r}\Zy\frac{(-1)^{j-(k+1)}}{(1-\mathbf{y}^{\frac{-\gamma}{M} })\prod\limits_{i=1}^{k}(1-\mathbf{y}^{\frac{\alpha_i}{M} })
     \prod\limits_{i=k+1}^{j-1}(1-\mathbf{y}^{\frac{-\alpha_{i}}{M} })\prod\limits_{i=j+1}^{n}(1-\mathbf{y}^{\frac{\alpha_{i}}{M} })}.
  \end{align}

Let $\eta_1=(-1,\ldots,-1,\gamma),\eta_2=(-1,\ldots,-1,-\gamma)$. Since
\begin{align*}
  &\sigma(\cone(A[(1,\ldots,j)\to\eta_1]))=\Zy\frac{(-1)^{j-1}}{(1-\mathbf{y}^\frac{\gamma}{M})\prod\limits_{i=1}^{j-1}(1-\mathbf{y}^{\frac{-\alpha_{i}}{M} })
  \prod\limits_{i=j+1}^{n}(1-\mathbf{y}^{\frac{\alpha_{i}}{M} })},\\
  &\sigma(\cone(A[(k+1,\ldots,j)\to\eta_2)]))=\Zy\frac{(-1)^{j-(k+1)}}{(1-\mathbf{y}^{\frac{-\gamma}{M} })\prod\limits_{i=1}^{k}(1-\mathbf{y}^{\frac{\alpha_i}{M} })
     \prod\limits_{i=k+1}^{j-1}(1-\mathbf{y}^{\frac{-\alpha_{i}}{M} })\prod\limits_{i=j+1}^{n}(1-\mathbf{y}^{\frac{\alpha_{i}}{M} })},
\end{align*}
 the theorem follows.
\end{proof}

\section{PDBarv: A Primal-Dual Barvinok Algorithm }
The \emph{primitive reduction} of a nonsingular matrix \( A \in \mathbb{Q}^{d \times d} \) is defined as the unique integer matrix whose columns are primitive vectors that generate the same cone as those of \( A \). The matrix $A$ is called \emph{primitive} if it equals its primitive reduction. In this case, the cone \( K = \cone(A) \) has index \( \ind(K) = \left|\det(A)\right| \).

According to duality theory~\cite{Polarity-of-cone}, the matrix \( (A^{-1})^T \) consists of the generating vectors for the dual cone of \( K \). Denote its  primitive reduction by \( A^* \) and let \( K^* = \cone(A^*) \). It is convenient to use the group $\{\id, *\}$, where the identity element $\id$ means primal space and $*$ means dual space, with $*\cdot *=\id$.
\subsection{Algorithm Design}

Theorem~\ref{core_thm} provides a decomposition formula for the cone $K$ without lower dimensional cones. This formula is inherently restricted to the primal space. In the dual space, we utilize Lemma~\ref{l-cone-dec} for decomposition. To enhance computational efficiency, we adopt a strategic approach: if the index of $K$ satisfies $\ind(K) \leq \ind(K^*)$, we decompose $K$; otherwise, we decompose its dual cone $K^*$. This strategy applies to any intermediate simplicial cone $K$. By integrating the decomposition formulas with this adaptive strategy, we propose the PDBarv algorithm (short for primal-dual Barvinok) for computing the generating function of a simplicial cone. This approach improves the efficiency of cone decomposition, particularly in special cases.

We present a pseudo-code as follows,  where $\Gamma(A)$ (currently treated as $A$) will be discussed later.

\algtext*{EndIf}
\algtext*{EndWhile}
\algtext*{EndFor}
\begin{algorithm}[H]
\caption{PDBarv Cone Decomposition}
\begin{algorithmic}[1]
\Require  A  nonsingular primitive matrix $B\in\mathbb{Z}^{d\times d}$ generating $K= \cone(B)$.
\Ensure   A list of $(\epsilon_i,\sigma(\cone(B_i)))$ as described in Equation \eqref{sign_sum_smallcones}.
    \State Set two lists Uni $= []$ and NonUni $= [(1,\id, \Gamma(B))]$
    \While{NonUni $\neq []$}
        \State Pop and get the last element $(\epsilon_A,*_A, \Gamma(A))$ from NonUni, where  $\epsilon_A=\pm1$ and $*_A\in\{\id,*\}$.
        \If{$\ind(\cone(A))=1$}
            \State Add $(\epsilon_A,\sigma(\cone(A^{*_A})))$ to Uni.
        \Else
            \If{$\ind(\cone(A))> \ind(\cone(A^{*}))$}
            \State Update $(\epsilon_A,*_A, \Gamma(A))$  by $(\epsilon_A,*\cdot*_A, \Gamma(A^*))$.
            \EndIf
            \State Compute the smallest vector $\beta = (k_1, \ldots, k_d)^T$ in $L(A^{-1})$ and $\gamma=A\beta$.
            \If{$*_A=\text{id}$} \hspace{1cm}\# we are in the primal space
            \State Permute the columns of $A$ to apply Theorem \ref{core_thm}, and add every $(s_i\cdot\e_A,*_{A}, \Gamma(A_i))$ to NonUni.
            \Else \hspace{1cm}\# $*_A=*$, so we are in the dual space
                \If{$k_i \leq 0$ for all $i$}
                 \State Update $\beta=-\beta$ and $\gamma=A\beta$.
                 \EndIf
                 \State   Apply Lemma~\ref{l-cone-dec} to add $(\text{sgn}(k_i)\cdot\e_A,*_{A}, \Gamma(A_i))$ to NonUni.
            \EndIf
        \EndIf

    \EndWhile
 \Return Uni.
\end{algorithmic}\label{alg-PDB}
\end{algorithm}

\begin{theorem}
Algorithm~\ref{alg-PDB} correctly computes the unimodular decomposition as in Equation \eqref{sign_sum_smallcones}.
\end{theorem}
\begin{proof}[Proof]
If lines 7 and 8 are removed from the PDBarv algorithm, we obtain the traditional Barvinok algorithm (except that we use Theorem~\ref{core_thm}
here), or dual Barvinok algorithm if we start with $B^*$.

In line 3, we pick out the last element $(\epsilon_A,*_A, \Gamma(A))$ in NonUni. We are dealing with a cone
in the primal space when $*_A=\id$ or in the dual space when $*_A=*$.

If $\ind(A)=\ind(A^*)=1$, we obtain an unimodular cone, and the corresponding signed cone is added to Uni.
These steps are performed in lines 4 and 5.

If $\ind(A)>1$, we check if $\ind(A)>\ind(A^*)$ in line 7. If true, then
we switch from primal space to dual space or vice versa in line 8. This corresponds to replacing $(*_A,\Gamma(A))$ with $(*_A\cdot *,\Gamma(A^*))$.

In lines 9--15, Theorems~\ref{core_thm} and Lemma~\ref{l-cone-dec} are respectively applied to recursively reduce the index of the simplicial cone, thereby justifying Algorithm~\ref{alg-PDB}.
\end{proof}

\subsection{Complexity Result and $\Gamma(A)$}
The PDBarv algorithm has two variants: PBarv (primal Barvinok) and DBarv (dual Barvinok). The former omits Steps 7--8, while the latter initializes with $\text{NonUni} = [(1, *, \Gamma(B^*))]$ and similarly skips Steps 7--8. Our subsequent analysis focuses on PBarv due to its structural parallels with DBarv.

The polynomial complexity of PBarv implies polynomial behavior for PDBarv, as the latter may accelerate index reduction through cone shortcuts. The computational bottleneck occurs in the iteration loop  (Steps 7--15), particularly in the implementation of the LLL algorithm in Step 9. To ensure theoretical polynomial complexity,  one needs to employ \emph{the enumeration step} to identify the smallest vector under the $\|\cdot \|_\infty$ norm. However, this step is rarely used in practical implementations. For further details, refer to \cite{De-loear-LLLwork-finish-Bar}.

Next we analyze the speed of the $w$Barv algorithms, where $w\in \{\text{P, D, PD}\}$,  for a concrete $K=\cone(B)$.
Denote by $n^w(K)$ the number of cones produced by the $w$Barv algorithm. Denote by $t^w(K)$ the runtime of the $w$Barv algorithm.
Then $t^w(K)$ is almost linear in $n^w(K)$. The reason is two-folded: the LLL algorithm has an average-case complexity of $O(d^3\log d)$ \cite{cl2013LLL} for
lattices of dimension $d$; the number of calls of LLL is almost linear in $n^w(K)$.

From the above view, we conclude that PDBarv outperforms PBarv or DBarv because PDBarv has potential shortcuts to decrease the index, and hence to produce a smaller number of cones.

\subsection*{Introduction of $\Gamma(A)$}

In extremal cases where Step 7 consistently returns false, PDBarv differs from PBarv only through the additional $\ind(A^*)$ computation in Step 7. This introduces an $O(d^3)$ overhead, potentially making PDBarv a bit slower.
However, this discrepancy becomes negligible when employing $\Gamma(A)$, defined as

\[
\Gamma(A) = \left(A,\ \det(A),\ A^*,\ \det(A^*),\ G\right)
\]
where $G$ is the unique matrix satisfying $A^* \cdot G = \adj(A)^T$. We use the adjugate matrix $\adj(A)$ rather than $A^{-1}$ here to avoid the costly rational arithmetic.

From the relationships between $A^*$, $(A^{-1})^T$, and $\adj(A)^T$, it follows that:
\begin{itemize}
    \item When $\det(A) > 0$, $G$ is the diagonal matrix whose diagonal entries are the greatest common divisors of the corresponding columns of $\adj(A)^T$.
    \item When $\det(A) < 0$, $-G$ is the diagonal matrix whose diagonal entries are the greatest common divisors of the corresponding columns of $\adj(A)^T$.
\end{itemize}
Crucially, the relationship $\cone(A^*) = \cone((A^{-1})^T)$ enables efficient computation since $A^{-1}$ is already utilized in Step 9.

The connection between $A^{-1}$ and $A[(i\to\gamma)]^{-1}$ was established in \cite{ORTextbookCompilationGroup1990}, enabling efficient computation of $\adj(A([i\to \gamma]))$ from $\adj(A)$. We formalize this relationship through the following lemma.

\begin{lemma}\label{PDBarv_lemma}
Let  $ A = (\alpha_1, \alpha_2, \ldots, \alpha_d) $  be a nonsingular  $ d \times d $  matrix. For any column vector  $ \gamma = \sum_{j=1}^d k_j \alpha_j $  and a fixed index  $i$  such that  $k_i \neq 0 $ , the adjugate matrix of the column-modified matrix satisfies
\begin{equation}
\adj(A[i \to \gamma]) = k_i (e_1, e_2, \ldots, \xi, \ldots, e_d) \adj(A),
\end{equation}
where  $ \xi = \left(-\frac{k_1}{k_i}, -\frac{k_2}{k_i}, \ldots, \frac{1}{k_i}, \ldots, -\frac{k_d}{k_i}\right)^T $  and  $ \{e_j\}_{j=1}^d $  denote the standard basis vectors.
\end{lemma}

To facilitate complexity analysis, let $A = (a_{ij})_{d \times d}$ be a nonsingular matrix and define $\kappa(A) = \prod_{j=1}^d \min \{ |a_{ij}| \neq 0 : i = 1,\ldots,d \}$. The following corollary shows that in extreme cases, this deficiency is negligible.
\begin{corollary}\label{complexana}
Let  $ A $  be as above, and let  $ \gamma = \sum_{j=1}^d k_j \alpha_j $  with a fixed  $i $  such that  $ k_i \neq 0 $. Given  $ \Gamma(A) $ , the following computational complexities hold:
\begin{itemize}
    \item $\Gamma(A^*)$ can be computed using $O(d)$ ring operations
    \item $\Gamma(A[(i\to \gamma)])$ can be computed using $O(d\log m_1) + O(d^2)$ ring operations
\end{itemize}
where $m_1 = \kappa(\adj(A[(i\to\gamma)])^T)$.
\end{corollary}

\begin{proof}[Proof]
For $\Gamma(A^*) = (A^*, \det(A^*), A, \det(A), G_1)$, the critical computation lies in $G_1$. From the adjugate matrix property:
\[
(A^* \cdot G) \cdot A^T = \det(A) \cdot I_d
\]
We derive through matrix normalization:
\[
A^* \cdot \left(G \cdot A^T \cdot \frac{\det(A^*)}{\det(A)}\right) = \det(A^*) \cdot I_d
\]
Thus $\adj(A^*) = G \cdot A^T \cdot \frac{\det(A^*)}{\det(A)}$. This identity yields $G_1 = G \cdot \frac{\det(A^*)}{\det(A)}$, establishing the $O(d)$ complexity for $\Gamma(A^*)$.

Let $B = A[(i\to\gamma)]$ with $B^* \cdot G_2 = \adj(B)^T$. The determinant relations:
\[
\det(B) = k_i \det(A), \quad \det(B^*) = \frac{(k_i \det(A))^{d-1}}{\det(G_2)}
\]
combine with Lemma~\ref{PDBarv_lemma} to show $\adj(B)^T$ requires $O(d^2)$ operations. The Euclidean algorithm computes $G_2$ in $O(d\log m_1)$ operations due to coefficient bounds in $m_1$, while $B^*$ needs $O(d^2)$ operations. The total complexity follows.
\end{proof}

Although we need to compute \(ind(A^*)\), we also obtain the adjugate matrix and the matrix determinant required for the subsequent steps of the LLL algorithm.

\subsection*{Acceleration via the $\|\cdot\|_\infty$-norm}

Let $\delta_1,\dots,\delta_d$ be the reduced basis for $L(A^{-1})$. The LattE software package finds the smallest vector $\beta$ among the $\delta_i$ under the $\|\cdot \|_\infty$-norm. Through empirical experiments, we observe that selecting the smallest vector $\beta'$ among $\delta_i$ with respect to the $\|\cdot\|_1$-norm generates fewer unimodular cones than the $\|\cdot\|_\infty$-norm approach. We provide two justifications for this phenomenon:

\begin{enumerate}
    \item When $|\det(A)|$ is large, the $k_i$ for $\beta$ is very small. The vector $\beta'$ remains admissible since $\|\beta'\|_\infty \leq d \|\beta\|_\infty$, which ensures rapid index reduction.

    \item If $|k_i| > 0.5$, we adjust $\beta'$ via $\beta' \pm e_i$ to enforce $|k_i| < 0.5$. This update corresponds to adding or subtracting the $i$-th column of $A$ to $\gamma'$ (see Line 9 in Alg.~\ref{alg-PDB}). Such adjustment guarantees at least a halving of the index. This idea applies to $\beta$, See Example \ref{exa-infty} below.
\end{enumerate}

\begin{example}\label{exa-infty} Let $A$ be the nonsingular matrix given by
$$A=\begin{pmatrix}
3 & 1 & -2 & 0 & 0 & -4 & -1 & 0 \\
9 & 0 & 7 & 0 & 2 & -10 & -5 & 0 \\
10 & 3 & 2 & 0 & 2 & -3 & -5 & -4 \\
-9 & -3 & 2 & 1 & -1 & 4 & 3 & 3 \\
-8 & -2 & -4 & 0 & -2 & 9 & 4 & 1 \\
-3 & -4 & 4 & 0 & -1 & -4 & 1 & 4 \\
-4 & -3 & 6 & 0 & 0 & 4 & 1 & 1 \\
-8 & -2 & -2 & -1 & -1 & -3 & 4 & 5
\end{pmatrix}.$$
Then the smallest vector $\beta$ is given by $\beta^T=\begin{pmatrix}
1/3 & -1/3 & 0 & -1/3 & 0 & -2/3 & 1/3 & -1/3
\end{pmatrix}$, which has $\|\beta\|_\infty=2/3>0.5$.
\end{example}

\subsection{Computational Experiments}\label{sec:expts}
In the practical implementation of the PDBarv, we have adopted the accelerations as mentioned above.

We present a computational comparison of the PDBarv algorithm against two established approaches: the Primal Irrational Barvinok algorithm~\cite{Primal-irrational} and the Dual Barvinok algorithm \cite{Dual-barvinok}, for generating function computation of cones. The latter two algorithms were implemented using version $1.7.5$ of the LattE software package.
To ensure a fair comparison, our implementation uses the same NTL library version as LattE, with consistent configurations for the LLL algorithm component.

\subsection*{Random experiments}

In this set of experiments, we randomly generate cones $K=\cone(A)$, where $A$ is a non-singular matrix whose entries are chosen uniformly from $[-100, 100]$ for dimensions $d$ between $4$ and $6$, and from $[-30, 30]$ for $d=7$.

Table~\ref{tab:decomp} summarizes the outcomes. For each experiment Ex.$i$ ($i=1,2,\dots, 11$), the matrix $A_i$ is stored in file Ex.$i$. The ``PDBarv'' column lists both the number of unimodular cones and the runtime produced by the algorithm; these data are identical whether decomposing $K_i$ or $K_i^*$. The column ``min-LattE'' displays the minimum between: (i) the result from applying the Primal Irrational algorithm to $K_i$, and (ii) the result from applying the Dual Barvinok algorithm to $K_i^*$ within the LattE implementation.

These experiments reveal two distinct efficiency regimes dictated by the index relationship.

 \begin{itemize}
\item When \(\ind(K) \leq \ind(K^*)\), PDBarv outperforms the Primal Irrational Barvinok algorithm,
      while both methods exhibit significantly higher efficiency than the Dual Barvinok algorithm.
\item When \(\ind(K) > \ind(K^*)\), PDBarv maintains clear superiority over the Dual Barvinok algorithm,
      whereas the Primal Irrational Barvinok algorithm demonstrates substantially slower performance.
 \end{itemize}

To eliminate potential anomalies, we performed multiple computational trials across various dimensions, carefully documenting the associated performance metrics. For each trial, we present a scatter plot in Figure~\ref{fig:performance}, where the $x$-axis represents the dimension $d$, and the $y$-axis corresponds to the ratio $\frac{\text{number of cones by min-LattE}}{\text{number of cones by PDBarv}}$.

\begin{table}[H]
    \centering
    \caption{Number of unimodular cones and runtime for PDBarv versus min-LattE across various dimensions.}\label{tab:decomp}
    \begin{tabular}{@{}lc>{\centering\arraybackslash}p{2.5cm}>{\centering\arraybackslash}p{2.5cm}>{\centering\arraybackslash}p{2.5cm}>{\centering\arraybackslash}p{2.5cm}@{}}
        \toprule
        Example & Dimension & \multicolumn{2}{c}{PDBarv (Alg.~\ref{alg-PDB})}  & \multicolumn{2}{c}{min-LattE} \\
        \cmidrule{3-6}
        & & Unimodular & Time(s) & Unimodular  & Time(s) \\
        \midrule
        Ex.1 & 4 & 1,504 & 0.08 & 1,885 & 0.08 \\
        Ex.2 & 4 & 1,739 & 0.09 & 2,007 & 0.08 \\
        Ex.3 & 4 & 2,392 & 0.12 & 2,940 & 0.12 \\
        Ex.4 & 5 & 43,547 & 2.8 & 55,107 & 3.3 \\
        Ex.5 & 5 & 59,478 & 4.0 & 73,168 & 4.4 \\
        Ex.6 & 5 & 81,060 & 4.9 & 111,955 & 6.7 \\
        Ex.7 & 6 & 1,734,263 & 150 & 2,608,010 & 238 \\
        Ex.8 & 6 & 3,047,166 & 260 & 5,281,072 & 453 \\
        Ex.9 & 6 & 4,064,428 & 345 & 6,084,859 & 527 \\
        Ex.10 & 7 & 8,859,812 & 801 & 15,940,441 & 1990 \\
        Ex.11 &7  & 9,164,539 & 852 & 16,084,606 & 1995 \\
        \bottomrule
    \end{tabular}
\end{table}

\begin{figure}[H]
    \centering
    \includegraphics[width=0.7\linewidth]{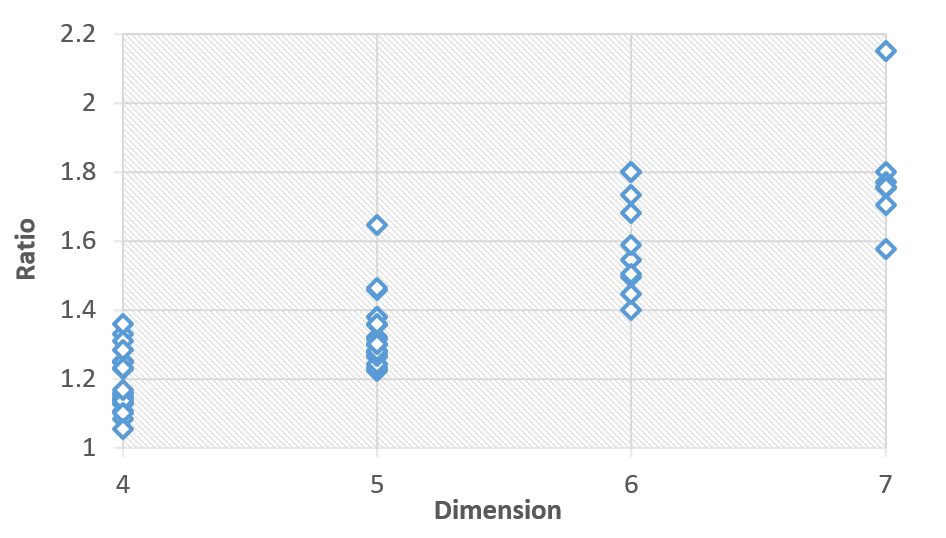}
    \caption{Cones cardinality ratios for PDBarv}
    \label{fig:performance}
\end{figure}

We primarily compare performance using the cardinality of the resulting unimodular cones.
As visually reinforced by Figure~\ref{fig:performance}, PDBarv demonstrates dimension-dependent performance improvements over LattE: approximately 15\% in dimension 4, 24\% in dimension 5, 37\% in dimension 6, and 44\% in dimension 7. The inherent $\|\cdot \|_1$-norm minimization in this metric explains the observed performance hierarchy.

\medskip
Based on our randomized tests, PDBarv almost always reduces to either PBarv or DBarv. Specifically, the alternation between primal and dual spaces occurs only at the tail ends of the decomposition. This behavior arises because the equality \(\ind(\cone(A^*)) = \ind(\cone(A))^{d-1}\) holds for most randomly generated \(A\) in our tests.

Nevertheless, in specific parameter regimes, such alternation occurs and yields acceptable performance improvements over PBarv, with both methods employing the \(\|\cdot \|_1\)-norm.

To illustrate, consider three full-dimensional simplicial cones \(K_1 = \cone(A_1)\), \(K_2 = \cone(A_2)\), and \(K_3 = \cone(A_3)\), where \(A_1\), \(A_2\), and \(A_3\) are nonsingular matrices:
   $$A_1=\begin{pmatrix}
3 & 1294707377 & 83497361 & 24919 & 135 \\
18 & 7200713520 & -127285340 & -25220 & -546 \\
-4 & -1294707377 & 43787979 & 301 & 411 \\
-39 & 1294707377 & 83497361 & 50139 & 499 \\
0 & 0 & 127285340 & 25220 & 728 \\
\end{pmatrix},$$
\begin{align*}
   & A_2=\begin{pmatrix}
5 & 8196010307 & 5929589 & 134059 & 319 \\
-20 & 11498526310 & 0 & 0 & 0 \\
-49 & 0 & 25609190 & 0 & 0 \\
16 & 0 & 0 & 135978 & 0 \\
15 & 0 & 0 & 0 & 346 \\
\end{pmatrix},  \\
   & A_3=\begin{pmatrix}
35 & 10636123 & 204109 & 4559 & 809 & 23 \\
33 & 53578602 & 0 & 0 & 0 & 0 \\
25& 0 & 695826 & 23994 & 558 & 18 \\
24& 0 & 0 & 35991 & 0 & 0 \\
32& 0 & 0 & 0 & 837 & 0 \\
11& 0 & 0 & 0 & 0 & 27 \\
\end{pmatrix}.
\end{align*}
The computation data is presented as follows,
$$
\begin{array}{r|r|r}
\text{cone} & \text{number of cones for PDBarv} & \text{number of cones for PBarv}\\\hline
K_1 & 285,197  & 369,484   \\
K_2 & 612,161  & 703,272  \\
K_3 & 4,768,283  & 5,190,608
\end{array},
$$
which shows the advantages of our PDBar algorithm.
\section{Concluding Remark}
We introduce a novel cone decomposition formula. This formula yields a more concise representation while enhancing geometric intuition and simplifying algorithmic implementation. Building on this foundation, we develop the PDBar algorithm, which strategically chooses to decompose in either the primal space or the dual space based on the minimum of the index of a cone and that of its dual. Owing to the explicit form of the formula, we anticipate applications to parametric simplicial cones.

We observe that in identifying suitable short vectors $\gamma$, the strategy employing the $\|\cdot\|_1$-norm outperforms that using the $\|\cdot\|_\infty$-norm. This observation is supported by our computer experiments; however, a proof remains elusive.

\section{Appendix}
\begin{appendix*}[Restatement of Proposition \ref{int_Dec})] \label{appendix1}
Let $A = (\alpha_1, \alpha_2, \ldots, \alpha_n)$ generate an $n$-dimensional rational simplicial cone $K = \cone(A)$ in $\mathbb{R}^d$. For any point $\gamma = \sum_{i=1}^r k_i\alpha_i \in K$ with $k_i > 0$, there exists a decomposition:
\begin{equation*}
    \sigma(K) = \sum_{i=1}^r \sigma(K_i),
\end{equation*}
where each subcone is defined as $K_i = \cone^{[i-1]}\left(A[(i \to \gamma)]\right)$.
\end{appendix*}

\begin{proof}[Proof]
We first normalize $\gamma$ to $\sum_{i=1}^r \alpha_i$ by scaling the generators appropriately. Consider an arbitrary point:
\[
\beta = \sum_{i=1}^n t_i\alpha_i \in K \quad \text{with } t_i \geq 0.
\]
Define $t_\ell := \min\{t_1, \ldots, t_r\}$ and select the smallest index $\ell$ achieving this minimum. This induces a partition:
\[
K = \bigsqcup_{\ell=1}^r C_\ell,
\]
where $C_\ell$ consists of points where $t_\ell$ is the first minimal coefficient among $t_1, \ldots, t_r$. Specifically:
\[
C_\ell = \left\{ \sum_{i=1}^n t_i\alpha_i \;\Big|\; t_j > t_\ell \text{ for } j < \ell,\ t_j \geq t_\ell \text{ for } \ell<j \leq r \right\}.
\]

Rewriting $\beta \in C_\ell$ using the modified generator $\gamma$:
\begin{align*}
\beta &= \sum_{j=1}^{\ell-1} (t_j - t_\ell)\alpha_j + \sum_{j=\ell+1}^r (t_j - t_\ell)\alpha_j + t_\ell\gamma + \sum_{s=r+1}^n t_s\alpha_s.
\end{align*}
This expression reveals that $C_\ell$ corresponds precisely to the half-open cone $K_\ell = \cone^{[\ell-1]}(A[(\ell \to \gamma)])$.
The disjoint union ensures the desired decomposition.
\end{proof}

\begin{appendix*}[Restatement of Lemma \ref{thm_lowd_dec})]\label{appendix2}
Let $A = (\alpha_1, \ldots, \alpha_n)$ generate an $n$-dimensional rational simplicial cone $K = \cone(A)$. Then
\begin{equation}
    \sigma(K) = \sum_{i=1}^{r-1} \sigma(F_i) + \sigma\left(\cone^{[r-1]}(A)\right),
\end{equation}
where each face $F_i$ is defined as $F_i = \cone^{[i-1]}\big(A[(i \to 0)]\big)$ for $1 \leq i \leq r-1$.
\end{appendix*}

\begin{proof}[Proof]
It suffices to prove that
\begin{equation}
    K \setminus \cone^{[r-1]}(A) = \bigsqcup_{i=1}^{r-1} F_i.
\end{equation}

For any point $\beta = \sum_{j=1}^n t_j\alpha_j$ in this complement set, let $i_0(\beta)$ denote the smallest index $i \leq r-1$ satisfying $t_i = 0$. This indexing induces a partition of the complement into disjoint faces:
\begin{itemize}
    \item $F_1$ contains points with $t_1 = 0$ and $t_j > 0$ for $j < 1$ (vacuously true)
    \item $F_i$ contains points with $t_i = 0$ and $t_j > 0$ for $1 \leq j < i$ when $i \geq 2$
    \item All $t_k \geq 0$ for $k > i$
\end{itemize}

The disjointness follows from the uniqueness of the first vanishing coordinate $i_0(\beta)$. Since the signature function $\sigma$ is additive over disjoint measurable sets, the decomposition holds by construction.
\end{proof}

\newpage
\bibliographystyle{IEEEtran}
\small\bibliography{reference}

\end{document}